\documentclass[12pt]{amsart}
\usepackage{SSdefn}

\DeclareMathOperator{\hgt}{ht}
\DeclareMathOperator{\Sat}{Sat}

\newcommand{\stacks}[1]{\cite[\href{http://stacks.math.columbia.edu/tag/#1}{Tag~#1}]{stacks}}

\title{Relative big polynomial rings}

\author{Andrew Snowden}
\address{Department of Mathematics, University of Michigan, Ann Arbor, MI}
\email{\href{mailto:asnowden@umich.edu}{asnowden@umich.edu}}
\urladdr{\url{http://www-personal.umich.edu/~asnowden/}}

\thanks{The author was supported by NSF DMS-1453893.}

\date{\today}

\begin{document}

\begin{abstract}
Let $K$ be the field of Laurent series with complex coefficients, let $\cR$ be the inverse limit of the standard-graded polynomial rings $K[x_1, \ldots, x_n]$, and let $\cR^{\flat}$ be the subring of $\cR$ consisting of elements with bounded denominators. In previous joint work with Erman and Sam, we showed that $\cR$ and $\cR^{\flat}$ (and many similarly defined rings) are abstractly polynomial rings, and used this to give new proofs of Stillman's conjecture. In this paper, we prove the complementary result that $\cR$ is a polynomial algebra over $\cR^{\flat}$.
\end{abstract}

\maketitle
\tableofcontents

\section{Introduction}

\subsection{Statement of results}

Let $A=\bC\lbb t \rbb$ be the ring of power series over the complex numbers and let $K=\Frac(A)$ be the field of Laurent series. The following rings are the main players in this paper:
\begin{itemize}
\item  Let $\cR$ be the inverse limit of the standard-graded polynomial rings $K[x_1, \ldots, x_n]$ in the category of graded rings. Thus $\cR$ is a graded ring, and a degree $d$ element of $\cR$ is a formal $K$-linear combination of degree $d$ monomials in the variables $\{x_i\}_{i \ge 1}$.
\item Let $\cR^+$ be the subring of $\cR$ with coefficients in $A$.
\item Let $\cR^0$ be the subring of $\cR$ with coefficients in $\bC$.
\item Let $\cR^{\flat}$ be the subring of $\cR$ where the coefficients have bounded denominators, i.e., $f \in \cR^{\flat}$ if and only if there is some $n$ such that $t^n f \in \cR^+$. Thus $\cR^{\flat}=\cR^+[1/t]$.
\end{itemize}
In \cite{stillman}, we showed that both $\cR$ and $\cR^{\flat}$ are abstractly polynomial algebras over $K$, and used these results to give new proofs of Stillman's conjecture. Given these results, it is natural to ask if $\cR$ is isomorphic to a polynomial algebra over its subring $\cR^{\flat}$. Our main result is that this is indeed the case:

\begin{theorem} \label{mainthm}
The ring $\cR$ is a polynomial algebra over $\cR^{\flat}$. More precisely, the map
\begin{equation} \label{eq:thm}
\cR^{\flat}_+/(\cR^{\flat}_+)^2 \to \cR_+/\cR_+^2
\end{equation}
is injective. Suppose that $\{\xi_i\}_{i \in I}$ are homogeneous elements of $\cR_+$ whose images form a basis of $\cR_+/(\cR^{\flat}_+ + \cR^2_+)$. Then the $\cR^{\flat}$-algebra homomorphism $\cR^{\flat}[X_i]_{i \in I} \to \cR$ mapping $X_i$ to $\xi_i$ is an isomorphism of graded $\cR^{\flat}$-algebras.
\end{theorem}

Recall from \cite{ananyan-hochster} that a homogeneous element $f$ of a graded ring $R$ has \emph{strength} $\le n$ if there is an expression $f = \sum_{i=1}^n g_i h_i$ where the $g_i$ and $h_i$ are homogeneous elements of positive degree. If no such expression exists, we say that $f$ has strength $\infty$. The ideal $R_+^2$ is exactly the ideal of finite strength elements. Thus the injectivity of \eqref{eq:thm} (which is the only non-trivial part of the theorem) is equivalent to the following, which is what we actually prove:

\begin{theorem} \label{mainthm2}
Let $f$ be an element of $\cR^{\flat}$ that has finite strength in $\cR$. Then $f$ has finite strength in $\cR^{\flat}$.
\end{theorem}

Here is the idea of the proof. Let $f$ be a given element of $\cR^{\flat}$ that has finite strength in $\cR$. Scaling by a power of $t$, we can assume that $f \in \cR^+$. Let $I$ be the ideal of $\cR^+$ generated by the partial derivatives of $f$. We show that the extension of $I$ to $\cR^{\flat}$ is contained in the extension of some finitely generated ideal $J \subset \cR^+$. This follows from elementary arguments involving heights, combined with the polynomiality of $\cR^{\flat}$ and $\cR$. It follows that $I$ is contained in the $t$-adic saturation of $J$. The main technical result of this paper (Theorem~\ref{thm:sat}) shows that such a saturation is (close enough to) finitely generated. From here, another elementary argument shows that $f$ has finite strength in $\cR^+$.

\subsection{Additional results on $\cR^+$}

As mentioned, the rings $\cR$ and $\cR^{\flat}$ are polynomial $K$-algebras. It is therefore easy to prove all sorts of results about heights in these rings. The ring $\cR^+$, on other hand, is \emph{not} a polynomial $A$-algebra: indeed, if it were then its graded pieces would be free $A$-modules, but its graded pieces are infinite products of $A$, which are not free. It is therefore not obvious how heights behave in $\cR^+$.

In the course of this work, we discovered a number of results about heights in $\cR^+$, such as a version of the Hauptidealsatz (Proposition~\ref{prop:hauptideal}) and a form of the catenary property (Proposition~\ref{prop:cat}). Although these results are not needed to prove the main theorem, we have included them in \S \ref{s:Rplus} as they use closely related methods.

\subsection{Motivation}  

There are two sources of motivation for this work. One comes from commutative algebra. As mentioned, our polynomiality results for rings like $\cR$ and $\cR^{\flat}$ were used in \cite{stillman} to give two new proofs of Stillman's conjectures, following the original proof in \cite{ananyan-hochster}. Shortly thereafter, our polynomiality results were used in \cite{draisma-lason-leykin} to give a fourth proof of Stillman's conjecture. In \cite{imperfect}, we strengthened our polynomiality results, which allowed us to strengthen the results of \cite{ananyan-hochster} on small subalgebras. Due to these applications, we believe it is worthwhile to try to better understand the precise nature and extent of the polynomiality phenomena. Theorem~\ref{mainthm} is a step in this direction.

The second source of motivation comes from infinite dimensional algebraic geometry. More precisely, in \cite{polygeom}, we study certain infinite dimensional algebraic varieties equipped with an action of $\GL_{\infty}$, and establish a number of nice properties in this situation (such as an analog of Chevalley's theorem). An important open problem remaining in \cite{polygeom} concerns the precise structure of image closures. Theorem~\ref{mainthm} was proved with this problem in mind. To see the connection, suppose that $F(X_1, \ldots, X_r)$ is a polynomial in $r$ variables that is homogeneous of degree $d$, where $X_i$ has degree $d_i$. Then $F$ defines a function
\begin{displaymath}
F \colon \cR^0_{d_1} \times \cdots \times \cR^0_{d_r} \to \cR^0_d.
\end{displaymath}
The space $\cR^0_d$ is an example of an infinite dimensional variety with $\GL_{\infty}$-action, as it can be identified with the dual of $\Sym^d(\bC^{\infty})$. Suppose $f \in \cR^0_d$. Theorem~\ref{mainthm} implies that if $f$ can be realized in the form $\lim_{t \to 0} F(g_1, \ldots, g_r)$ with $g_i \in \cR_{d_i}$, then it can be realized in this form with $g_i \in \cR^{\flat}_{d_i}$. In other words, if $f$ can be realized as a certain kind of ``wild'' limit in the image of $F$ then it can also be realized by a much nicer ``tame'' kind of limit. We had hoped to use this to resolve the open question in \cite{polygeom}. Unfortunately, it does not appear to be quite enough. However, we believe that Theorem~\ref{mainthm} could still be useful in studying similar problems.

\subsection{Open problems}

Here are some open problems raised by our work:
\begin{itemize}
\item In the setting of \S \ref{s:sat}, is it true that the saturation of a finitely generated ideal is finitely generated?
\item Let $A$ be an integral domain such that $K=\Frac(A)$ is perfect, let $R$ be the inverse limit of the graded rings $K[x_1, \ldots, x_n]$, and let $R^{\flat}$ be the subring where the denominators are bounded (i.e., $f \in R^{\flat}$ if $af$ has coefficients in $A$ for some non-zero $a \in A$). In \cite{stillman}, we showed that $R$ and $R^{\flat}$ are polynomial $K$-algebras. Is $R$ a polynomial algebra over $R^{\flat}$? This paper only addresses the special case where $A=\bC\lbb t \rbb$.
\item As mentioned, $\cR^+$ is not a polynomial ring. However, the results of \S \ref{s:Rplus} show that in some ways it behaves like a polynomial ring. Can this observation be sharpened, or made more precise?
\end{itemize}

\subsection{Outline}

In \S \ref{s:bg}, we give some general background on heights. In \S \ref{s:hgt-compar}, we prove a comparison result for heights in $\cR^{\flat}$ and $\cR$. In \S \ref{s:nak}, we prove a Nakayama-like lemma that will be used in our analysis of saturation. In \S \ref{s:sat}, we prove the main technical result of the paper (Theorem~\ref{thm:sat}) on saturations. Using this, we prove our main theorem in \S \ref{s:mainthm}. Finally, in \S \ref{s:Rplus}, we prove some additional results about $\cR^+$.

\subsection*{Acknowledgments}

We thank Dan Erman for comments on a draft of this paper.

\section{Background on heights} \label{s:bg}

Let $R$ be a ring. Recall that the \emph{height} of a prime ideal $\fp$, denoted $\hgt_R(\fp)$, is the maximal value of $n$ for which there exists a strict chain of primes $\fp_0 \subset \cdots \subset \fp_n=\fp$, or $\infty$ if there exist arbitrarily long such chains. The \emph{height} of an ideal $I$, denoted $\hgt_R(I)$, is defined as the minimum of $\hgt_R(\fp)$ over primes $\fp$ containing $I$; by convention, the height of the unit ideal is infinity. If $I$ is an ideal in a finite variable polynomial ring $R=F[x_1, \ldots, x_n]$, with $F$ a field, then $\hgt_R(\fp)$ is the codimension of the locus $V(I) \subset \bA^n_F$. We note that if $I \subset J$ then $\hgt_R(I) \le \hgt_R(J)$. In what follows, polynomial rings can have infinitely many variables.

\begin{proposition} \label{prop:primefg}
Let $R$ be a polynomial ring over a field. Then any finite height prime ideal is finitely generated.
\end{proposition}

\begin{proof}
See \cite[Proposition~3.2]{stillman}
\end{proof}

\begin{proposition} \label{prop:finhtgen}
Let $R$ be a polynomial ring over a field. Then an ideal has finite height if and only if it is contained in a finitely generated non-unital ideal.
\end{proposition}

\begin{proof}
Suppose $I$ has finite height. Then, by definition, $I$ is contained in a prime of finite height, which is finitely generated by Proposition~\ref{prop:primefg}.

Now suppose $I$ is contained in a finitely generated non-unital ideal, say $(f_1, \ldots, f_r)$. Each $f_i$ uses only finitely many variables, and so $(f_1, \ldots, f_r)$ is extended from a finite variable subring. Performing such an extension does not change height \cite[Proposition~3.3]{stillman}.
\end{proof}

\begin{proposition} \label{prop:ht-ext}
Let $R$ be a finite variable polynomial ring over a field $F$, let $E/F$ be a field extension, and let $S=E \otimes_F R$. Let $\fp$ be a prime of $S$ and let $\fq$ be its contraction to $R$. Then $\hgt_R(\fq) \le \hgt_S(\fp)$.
\end{proposition}

\begin{proof}
The natural homomorphism $E \otimes_F R/\fp \to S/\fq$ is surjective. We thus find
\begin{displaymath}
\dim(R/\fp) = \dim(E \otimes_F R/\fp) \ge \dim(S/\fq),
\end{displaymath}
where $\dim$ denotes Krull dimension. Since $\dim(R/\fp)=n-\hgt_R(\fp)$ and $\dim(S/\fq)=n-\hgt_S(\fq)$, the result follows.
\end{proof}

\begin{remark}
(a) We can actually have $\hgt_R(\fq)<\hgt_S(\fp)$. For instance, let $F=\bC$ and $E=\bC(t)$, and take $\fp$ to be the ideal generated by $x_1-tx_2$. Then $\fp$ has height~1 but $\fq=0$ has height~0. (b) The proposition holds in the infinite variable case too.
\end{remark}

\begin{proposition} \label{prop:supht}
Let $R$ be a ring. Suppose the following condition holds:
\begin{itemize}
\item[$(\ast)$] If $I$ is a finitely generated ideal of height $c<\infty$, then there are only finitely many primes $\fp$ of height $c$ that contain $I$.
\end{itemize}
Let $I$ be an ideal of $R$ and let $I=\bigcup_{\alpha \in \cI} J_{\alpha}$ be a directed union, where $J_{\alpha}$ are ideals contained in $I$. Then $\hgt_R(I) = \sup_{\alpha \in \cI} \hgt_R(J_{\alpha})$.
\end{proposition}

\begin{proof}
For any $J \subset I$ we have $\hgt_R(J) \le \hgt_R(I)$, and so $\sup_{\alpha \in \cI} \hgt_R(J_{\alpha}) \le \hgt_R(I)$. We now prove the reverse inequality. First, suppose that the $J_{\alpha}$ are finitely generated. If $\sup_{\alpha} \hgt_R(J_{\alpha})$ is infinite then there is nothing to prove, so suppose it is a finite number $c$. Passing to a cofinal subset, we may as well suppose that $\hgt_R(J_{\alpha})=c$ for all $\alpha$. For each $\alpha$, let $\cP_{\alpha}$ be the set of prime ideals of height $c$ containing $J_{\alpha}$; this set is finite by hypothesis. Of course, if $\alpha \le \beta$ then $\cP_{\beta} \subset \cP_{\alpha}$. By a standard compactness result, we have $\bigcap_{\alpha \in \cI} \cP_{\alpha} \ne \emptyset$. We can thus find a prime $\fp$ of height $c$ such that $J_{\alpha} \subset \fp$ for all $\alpha$. Since $I=\bigcup_{\alpha \in \cI} J_{\alpha}$, we thus find $I \subset \fp$, and so $\hgt_R(I) \le c$.

We now treat the general case. For each $\alpha$, let $\{J_{\alpha,\beta}\}_{\beta \in \cK_{\alpha}}$ be the finitely generated ideals contained in $J_{\alpha}$. Then
\begin{displaymath}
\hgt_R(I) = \sup_{\alpha,\beta} J_{\alpha,\beta} = \sup_{\alpha} \sup_{\beta} \hgt(J_{\alpha,\beta}) = \sup_{\alpha} \hgt_R(J_{\alpha}),
\end{displaymath}
where in the first and last step we used the previous case.
\end{proof}

\begin{proposition} \label{prop:polystar}
Suppose $R$ is a polynomial ring over a field. Then condition $(\ast)$ holds.
\end{proposition}

\begin{proof}
Let $I$ be a finitely generated ideal of $R$ of height $c$. Then $I$ is extended from an ideal $I'$ of some finite variable subring $R_0$ of $R$. Since such extensions do not affect height \cite[Proposition~3.3]{stillman}, we see that $I'$ has height $c$ also. Let $\fq_1, \ldots, \fq_r$ be the minimal primes above $I'$ in $R_0$. Now, let $\fp$ be height $c$ prime of $R$ containing $I$. Then $\fp$ is finitely generated by Proposition~\ref{prop:primefg}, and thus extended from a prime $\fp'$ of a finite variable subring $R_1$ of $R$, which we can assume contains $R_0$. The ideal $I'R_1$ has height $c$ and $\fq_1R_1, \ldots, \fq_rR_1$ are the minimal primes over it. Since $\fp'$ contains $I'R_1$ and has the same height, it follows that $\fp'=\fq_iR_1$ for some $i$, and thus $\fp=\fq_iR$. We thus see that there are only $r$ choices for $\fp$.
\end{proof}

\begin{proposition} \label{prop:htcontract}
Let $A$ be a polynomial ring over a field and let $I$ be a finite height ideal of $A[x]$. Then $\hgt_A(I \cap A) \le \hgt_{A[x]}(I)$.
\end{proposition}

\begin{proof}
If $I$ is prime then it is finitely generated (Proposition~\ref{prop:primefg}) and the result follows from the corresponding result for finite variable polynomial rings and the fact that extending to larger polynomial rings does not change height \cite[Proposition~3.3]{stillman}. Now suppose that $I$ is a general ideal of finite height $c$. Let $\fp$ be a height $c$ prime containing $I$. Then $I \cap A \subset \fp \cap A$, and the latter has height $\le c$. Thus the result follows.
\end{proof}

Let $R$ be a ring and let $S$ be a multiplicative subset. A basic theorem states that the primes of $S^{-1} R$ correspond bijectively (via extension and contraction) to the primes of $R$ disjoint from $S$. The following result shows that this correspondence preserves height.

\begin{proposition} \label{prop:localht}
Let $\fp$ be a prime of $S^{-1}R$ and let $\fq$ be its contraction to $R$. Then $\hgt_R(\fq)=\hgt_{S^{-1}R}(\fp)$.
\end{proposition}

\begin{proof}
Let $\fp_0 \subset \cdots \subset \fp_r=\fp$ be a strict chain of primes in $S^{-1} R$. Contracting gives a strict chain of primes in $R$. Thus $\hgt_R(\fq) \ge \hgt_{S^{-1}R}(\fp)$.

Now let $\fq_0 \subset \cdots \subset \fq_r=\fq$ be a strict chain of primes in $R$. Since $\fq$ is the contraction of an ideal from $S^{-1}R$, it is disjoint from $S$, and so all the ideals $\fq_i$ are as well. Thus extending gives a strict chain, and so $\hgt_{S^{-1}R}(\fp) \ge \hgt_R(\fq)$.
\end{proof}

\begin{corollary} \label{cor:localht}
Let $I$ be an ideal of $R$. Then $\hgt_R(I) \le \hgt_{S^{-1} R}(S^{-1} I)$.
\end{corollary}

\begin{proof}
Let $c=\hgt_{S^{-1} R}(S^{-1}I)$. If $c=\infty$ there is nothing to prove, so suppose $c$ is finite. Let $\fp$ be a height $c$ prime of $S^{-1} R$ containing $S^{-1} I$. Let $\fq$ be the contraction of $\fp$. Then $\hgt_R(\fq)=c$ by the proposition. Since $I \subset \fq$, we thus have $\hgt_R(I) \le c$.
\end{proof}

\section{Comparing heights in $\cR^{\flat}$ and $\cR$} \label{s:hgt-compar}

Let $\cR_{>n}$ be defined like $\cR$, but only using the variables $x_i$ with $i>n$. Similarly define $\cR^{\flat}_{>n}$, and so on. We have a natural isomorphism $\cR=\cR_{>n}[x_1,\ldots,x_n]$, and similarly for the other variants.

\begin{proposition} \label{prop:zerocontract}
Let $I$ be an ideal of $\cR$ of finite height. Then $I \cap \cR_{>n}=0$ for $n \gg 0$. Similarly for $\cR^{\flat}$ and $\cR^0$.
\end{proposition}

\begin{proof}
Let $f$ be a non-zero element of $I$. As in \cite[Lemma~4.9]{stillman}, we can find a linear change of variables $\gamma$ in finitely many of the $x$'s so that $\gamma(f)$ is monic in $x_1$. We then have $\hgt_{\cR_{>1}}(\gamma(I) \cap \cR_{>1})=\hgt_{\cR}(I)-1$. Indeed, if $I$ is fintiely generated, this is \cite[Corollary~3.8]{stillman} (which easily reduces to the finite variable case), while Propositions~\ref{prop:supht} and~\ref{prop:polystar} allow us to reduce to this case. It follows from Proposition~\ref{prop:htcontract} that $\hgt_{\cR_{>n}}(\gamma(I) \cap \cR_{>n})<\hgt_{\cR}(I)$ for all $n \ge 1$. Taking $n$ larger than the variables used in $\gamma$, we thus have $\hgt_{\cR_{>n}}(I \cap \cR_{>n})<\hgt_{\cR}(I)$. Thus, by induction on height, the result follows.
\end{proof}

\begin{remark}
The proposition does not hold for $\cR^+$: indeed, the ideal $(t)$ is a counterexample. We will formulate a version for $\cR^+$ in Proposition~\ref{prop:zerocontract2}.
\end{remark}

\begin{corollary}
Let $\fp$ be a prime ideal of $\cR$ of finite height. Then $\fp$ is the contraction of a prime ideal of $\Frac(\cR_{>n})[x_1,\ldots,x_n]$ for all sufficiently large $n$.
\end{corollary}

\begin{proposition} \label{prop:htbd1}
Let $I$ be an ideal of $\cR^{\flat}$. Then $\hgt_{\cR^{\flat}}(I) \le \hgt_{\cR}(I \cR)$.
\end{proposition}

\begin{proof}
Let $c=\hgt_{\cR}(I\cR)$. If $c=\infty$ there is nothing to prove, so suppose $c$ is finite. Let $\fp$ be a height $c$ prime of $\cR$ containing $I \cR$. Let $n$ be such that $\fp$ is contracted from a prime $\fp'$ of $\Frac(\cR_{>n})[x_1, \ldots, x_n]$, necessarily of height $c$ (by Proposition~\ref{prop:localht}). Let $\fq'$ be the contraction of $\fp'$ to $\Frac(\cR^{\flat}_{>n})[x_1, \ldots, x_n]$. Then $\fq'$ has height $\le c$ by Proposition~\ref{prop:ht-ext}. Let $\fq$ be the contraction of $\fq'$ to $\cR^{\flat}$. Then $\fq$ has the same height as $\fq'$ by Proposition~\ref{prop:localht}, which is $\le c$. Since $I$ is contained in $\fq$, it too has height $\le c$.
\end{proof}

\section{A Nakayama-like lemma} \label{s:nak}

Let $\Pi$ be an infinite product of copies of $A=\bC\lbb t \rbb$, and suppose that $M$ is an $A$-submodule of $\Pi$. We recall several concepts:
\begin{itemize}
\item $M$ is \emph{($t$-adically) complete} if the following holds: given elements $x_i \in t^i M$ for $i \ge 0$ the sum $\sum_{i \ge 0} x_i$ belongs to $M$.
\item $M$ is \emph{($t$-adically) closed} is the following holds: given elements $x_i \in M \cap t^i \Pi$ for $i \ge 0$ the sum $\sum_{i \ge 0} x_i$ belongs to $M$. Obviously, closed implies complete.
\item We let $\Sigma_n(M)$ be the set of elements $x \in \Pi$ such that $t^n x \in M$. It is an $A$-submodule of $\Pi$. We let $\Sigma(M) = \bigcup_{n \ge 1} \Sigma_n(M)$, which we call the \emph{saturation} of $M$.
\end{itemize}
The purpose of this section is to prove the following result:

\begin{proposition} \label{prop:nak}
Let $M$ be an $A$-submodule of $\Pi$. Suppose:
\begin{enumerate}
\item $M$ is complete.
\item $\Sigma_n(M)/M$ is finite dimensional over $\bC$ for all $n$.
\item $\Sigma(M)$ is contained in the closure of $M$.
\end{enumerate}
Then $M=\Sigma(M)$, that is, $M$ is saturated.
\end{proposition}

\begin{proof}
Let $\delta(n)$ be the dimension of $\Sigma_n(M)/M$. Let $x_1, x_2, \ldots$ be elements of $\Sigma(M)$ such that $x_1, x_2, \ldots, x_{\delta(n)}$ is a basis of $\Sigma_n(M)/M$ for each $n$. Note that $x_1, x_2, \ldots$ is a basis for $\Sigma(M)/M$.

Let $y$ be an element of the $t$-adic closure of $M$. We can thus write $y = \sum_{j=0}^{\infty} z_j$, where $z_j \in M \cap t^j \Pi$. Since $t^{-j} z_j \in \Sigma_j(M)$, we can write $t^{-j} z_j=b_j+c_{1,j} x_1 + \cdots + c_{a(j),j} x_{a(j)}$ with $b_j \in M$ and $c_{i,j} \in \bC$. We thus find
\begin{displaymath}
y
= \sum_{j \ge 0} t^j(b_j+c_{1,j} x_1 + \cdots + c_{a(j),j} x_{a(j)}) 
= \sum_{j \ge 0} t^j b_j + \sum_{i \ge 1} \sum_{\delta(j) \ge i} c_{i,j} t^j x_i
\end{displaymath}
Let $\epsilon(i)$ be the minimal value of $j$ such that $\delta(j) \ge i$, with the convention $\epsilon(i)=\infty$ if $\delta(j)<i$ for all $j$. Thus in the inner sum above, we can write $j \ge \epsilon(i)$. Note that $\epsilon(i) \ge 1$ for all $i$ and $\epsilon(i) \to \infty$ as $i \to \infty$.  Now, the first sum above belongs to $M$ by (a). And the inner sum in the second sum converges to an element of $A$ that is dividible by $\epsilon(i)$. We conclude that for any $y$ in the $t$-adic closure of $M$, we can write
\begin{displaymath}
y = b + \sum_{i \ge 0} f_i(t) x_i
\end{displaymath}
where $b \in M$ and $f_i(t) \in t^{\epsilon(i)} A$.

We now apply this to $x_i$, which belongs to the $t$-adic closure of $M$ by (c). We can thus write
\begin{displaymath}
x_i = b_i + \sum_{j \ge 0} f_{i,j} x_j
\end{displaymath}
where $b_i \in M$ and $f_{i,j} \in t^{\epsilon(j)} A$. Let $\ul{x}$ and $\ul{b}$ be the infinite column vectors $(x_1,x_2,\ldots)$ and $(b_1,b_2,\ldots)$, and let $A$ be the infinite square matrix given by $A_{i,j}=f_{i,j}$. We then get the linear equation
\begin{displaymath}
(1-A) \ul{x}=\ul{b}.
\end{displaymath}
The entries of $A$ all belong to the maximal ideal of $A$. In fact, for any $n$ there exists an $m$ such that all columns to the right of the $m$th column of $A$ are divisible by $t^n$. It follows that $1-A$ is invertible, and so the entries of $\ul{x}=(1-A)^{-1} \ul{b}$ belongs to $M$. Thus $x_i \in M$ for all $i$, which completes the proof.
\end{proof}

\section{A result on saturation} \label{s:sat}

Let $S$ be a graded polynomial $\bC$-algebra, where each variable is homogeneous of positive degree; the interesting case is where there are infinitely many variables. Let $R$ be the graded version of $S\lbb t \rbb$; thus $R$ is a graded ring and a degree $d$ element of $R$ is a power series in $t$ with coefficients in $S_d$. Note that elements of $R$ can use infinitely many of the variables in $S$, which is why $R$ can be hard to work with.

Given a homogeneous ideal $I$ of $R$, we define its \emph{saturation}, denoted $\Sat(I)$, to be the set of all elements $f \in R$ such that $t^n f \in I$ for some $n$. Thus $\Sat(I)$ is a homogeneous ideal with $\Sat(I)_n=\Sigma(I_n)$, where $\Sigma$ is as in the previous section. (We note that each graded piece of $R$ is isomorphic to a product of copies of $\bC\lbb t \rbb$.) We would like to know that the saturation of a finitely generated ideal is finitely generated. Unfortunately, we have not been able to prove this. However, we prove a weaker statement that is sufficient for our applications. For an integer $d \ge 0$, define $\Sat_{\le d}(I)$ to be the ideal generated by $I$ and $\Sat(I)_k$ for $0 \le k \le d$. The result is:

\begin{theorem} \label{thm:sat}
If $I$ is finitely generated then so is $\Sat_{\le d}(I)$, for any $d$.
\end{theorem}

Fix $I$ and $d$ as in the theorem. We may assume, by induction, that the theorem holds for smaller values of $d$, and so we may replace $I$ by $\Sat_{\le d-1}(I)$; this is still finitely generated by the inductive hypothesis. Thus if $f \in \Sat(I)$ has degree $<d$ then $f$ already belongs to $I$.

Let $f_1, \ldots, f_r$ be homogeneous generators of $I$. Let $X_k$ be the set of all tuples $(g_1, \ldots, g_r) \in R^r$ such that $\deg(g_if_i)=d$ for all $i$ and $t^k \mid g_1f_1+\cdots+g_rf_r$. Note that $X_{k+1} \subset X_k$. Define
\begin{displaymath}
\pi_k \colon X_k \to R_d, \qquad (g_1, \ldots, g_r) \mapsto \frac{g_1f_1+\cdots+g_rf_r}{t^k}.
\end{displaymath}
and put $Y_k=\im(\pi_k)$. Note that $Y_k$ is exactly $\Sigma_k(I_d)$, that is, the set of $g \in S_d$ such that $t^k g \in I$.

Let $E$ be the set of all tuples $(g_1, \ldots, g_r) \in S^r$ such that $g_1 f_1(0)+\cdots g_r f_r(0)=0$. Here $f_i(0) \in S$ is the result of substituting~0 for $t$ in $f_i$. Thus $E$ is just the syzygy module for $(f_1(0), \ldots, f_r(0))$. Let $\ol{E}=E_d/(E_d \cap S_+ E)$, i.e., the space of degree $d$ generators for $E$. Here $E_d$ consists of tuples $(g_1, \ldots, g_r) \in E$ such that $\deg(g_i f_i)=d$.

\begin{lemma}
$\ol{E}$ is a finite dimensional vector space.
\end{lemma}

\begin{proof}
Let $A$ be the set of variables appearing in $f_1(0), \ldots, f_r(0)$, which is finite. Let $S' \subset S$ be the polynomial ring in the $A$ variables, and define $E'$ like $E$, but using $S'$ instead of $S$. Since $S'$ is noetherian, it follows that $E'$ is a finitely generated module, and so $\ol{E}'=E'_d/(E'_d \cap S_+' E')$ is finite dimensional.

We have a natural inclusion $E' \to E$. We claim that the induced map $\ol{E}' \to \ol{E}$ is surjective, which will complete the proof. Thus suppose $(g_1, \ldots, g_r)$ is a degree $d$ element of $E$. Write $g_i = \sum_e g_{i,e} x^e$, where the sum is over all monomials in variables not in $A$, and $g_{i,e} \in S'$. Each tuple $(g_{1,e}, \ldots, g_{r,e})$ belongs to $E$ and so if $e \ne 0$ then $(x^e g_{1,e}, \ldots, x^e g_{r,e})$ belongs to $S_+ E$. Hence $(g_1, \ldots, g_r)$ and $(g_{1,0}, \ldots, g_{r,0})$ are equal in $\ol{E}$. Since the latter belongs to $E'$, the result follows.
\end{proof}

If $(g_1, \ldots, g_r)$ belongs to $X_1$ then $g_1 f_1 + \cdots + g_r f_r$ is divisible by $t$, which exactly means that $g_1(0)f_1(0)+\cdots+g_r(0) f_r(0)=0$, i.e., $(g_1(0), \ldots, g_r(0))$ belongs to $E_d$. Define $\rho \colon X_1 \to \ol{E}$ by taking $(g_1, \ldots, g_r)$ to the element represented by $(g_1(0), \ldots, g_r(0))$. Put $Z_k=\rho(X_k)$. Since $X_{k+1} \subset X_k$, we have $Z_{k+1} \subset Z_k$.

\begin{lemma}[Key Lemma]
Let $(g_1, \ldots, g_r) \in X_k$ with $k \ge 1$. Suppose that $\rho(g_1, \ldots, g_r)=0$. Then $\pi_k(g_1, \ldots, g_r) \in Y_{k-1}$.
\end{lemma}

\begin{proof}
Since $(g_1(0), \ldots, g_r(0))$ maps to~0 in $\ol{E}$, we can find elements $(a_{1,i}, \ldots, a_{r,i}) \in E$ for $1 \le i \le n$ of degree $<d$ such that
\begin{displaymath}
(g_1(0), \ldots, g_r(0)) = \sum_{i=1}^n b_i \cdot (a_{1,i}, \ldots, a_{r,i})
\end{displaymath}
for some $b_i \in S$. Now, we have
\begin{displaymath}
\frac{g_1(0) f_1 + \cdots + g_r(0) f_r}{t} = \sum_{i=1}^n b_i \cdot \frac{a_{1,i} f_1 + \cdots + a_{r,i} f_r}{t}.
\end{displaymath}
Since $a_{1,i} f_1 + \cdots + a_{r,i} f_r$ is divisible by $t$, its quotient by $t$ belongs to $\Sat(I)$. It also has degree $<d$, and so it belongs to $I$ by our initial setup. Thus we can write
\begin{displaymath}
b_i \cdot \frac{a_{1,i} f_1 + \cdots + a_{r,i} f_r}{t} = c_{1,i} f_1 + \cdots + c_{r,i} f_r
\end{displaymath}
for elements $c_{i,j} \in R$. Letting $c_i=c_{i,1} + \cdots + c_{i,r}$, we thus have
\begin{displaymath}
\frac{g_1(0)f_1+\cdots+g_r(0)}{t} = c_1f_1 + \cdots + c_r f_r.
\end{displaymath}
Now, write $g_i=g_i(0)+t g_i'$. Then
\begin{align*}
\pi_k(g_1, \ldots, g_r)
&= \frac{g_1 f_1+\cdots+ g_r f_r}{t^k} \\
&= \frac{\frac{g_1(0) f_1+\cdots+g_r(0) f_r}{t} + g_1' f_1+\cdots+g_r' f_r}{t^{k-1}} \\
&= \frac{(g_1'+c_1) f_1 + \cdots + (g_r'+c_r) f_r}{t^{k-1}} \\
&= \pi_{k-1}(g_1'+c_1, \ldots, g_r'+c_r).
\end{align*}
This completes the proof.
\end{proof}

\begin{lemma}
$Y_k/Y_{k-1}$ is finite dimensional over $\bC$.
\end{lemma}

\begin{proof}
The key lemma gives a surjection $Z_k \cong X_k/\ker(\rho \vert_{X_k}) \to Y_k/Y_{k-1}$, and $Z_k$ is finite dimensional.
\end{proof}

\begin{lemma}
Let $(g_1, \ldots, g_r) \in X_k$ with $k \ge 1$. Suppose that $\rho(g_1, \ldots, g_r) \in Z_{k+m}$. Then $\pi_k(g_1, \ldots, g_r) \in Y_{k-1}+t^m Y_{k+m}$.
\end{lemma}

\begin{proof}
Let $(h_1, \ldots, h_r) \in X_{k+m}$ be such that $\rho(g_1, \ldots, g_r)=\rho(h_1, \ldots, h_r)$. We have
\begin{displaymath}
\pi_k(h_1, \ldots, h_r)=t^m \pi_{k+m}(h_1, \ldots, h_r) \in t^m Y_{k+m}.
\end{displaymath}
We have
\begin{displaymath}
\pi_k(g_1, \ldots, g_r) = \pi_k(g_1-h_1, \ldots, g_r-h_r) + \pi_k(h_1, \ldots, h_r).
\end{displaymath}
The key lemma shows that the first term belongs to $Y_{k-1}$, while we have just seen that the second belongs to $t^m Y_{k+m}$.
\end{proof}

\begin{proof}[Proof of Theorem~\ref{thm:sat}]
The sets $Z_k$ form a descending chain in the finite dimensional space $\ol{E}$, and so they stabilize. Let $\ell$ be the point at which they stabilize, so that $Z_{\ell+n}=Z_{\ell}$ for all $n$. We claim that $Y_k=Y_{\ell}$ for all $k \ge \ell$, which will establish the result: indeed, $\Sat_{\le d}(I)$ will then be generated by $I$ and $Y_k$, and since $Y_k/Y_0=Y_k/I_d$ is finite dimensional, we are only adding finitely many generators to $I$. We prove this by applying Proposition~\ref{prop:nak} with $M=Y_{\ell}$. We check the three axioms:

(a) Suppose $x_0,x_1,\ldots \in Y_{\ell}$. We can thus write $x_i=t^{-\ell}(g_{1,i} f_1+\cdots+g_{r,i} f_r)$, and so $\sum_{i \ge 0} x_i t^i = t^{-\ell}(g_1f_1+\cdots+g_r f_r)$ where $g_j=\sum_{i \ge 0} g_{j,i} t^i$. This clearly belongs to $Y_{\ell}$.

(b) We have $\Sigma_k(Y_{\ell})=Y_{k+\ell}$, and we have already seen that $Y_{k+\ell}/Y_{\ell}$ is finite dimensional for all $k$.

(c) Let $(g_1, \ldots, g_r) \in X_k$ with $k>\ell$. Then $\rho(g_1, \ldots, g_r) \in Z_{k+m}$ for all $m$. Thus, by the previous lemma, we have $\pi_k(g_1,\ldots,g_r) \in Y_{k-1}+t^m Y_{k+m}$. We thus see that $Y_k \subset Y_{k-1}+t^m Y_{k+m} \subset Y_{k-1}+t^m R$. Applying this inductively, we find $Y_k \subset Y_{\ell}+t^m R$, for any $m \ge 0$. Thus $Y_k$ is contained in the $t$-adic closure of $Y_{\ell}$, for all $k \ge \ell$.

Proposition~\ref{prop:nak} now applies, and shows that $Y_{\ell}$ is saturated. Thus $Y_k=Y_{\ell}$ for all $k \ge \ell$, which establishes the result.
\end{proof}

\section{Proof of the main theorem} \label{s:mainthm}

Before proving the main theorem, we need the following simple result on strength.

\begin{proposition} \label{prop:str}
Let $f \in \cR$ be homogeneous. Then the following are equivalent:
\begin{enumerate}
\item $f$ has finite strength.
\item The ideal of $\cR$ generated by the partial derivatives of $f$ is contained in an ideal generated by finitely many homogeneous elements of positive degrees.
\end{enumerate}
The same statement holds for $\cR^+$ and $\cR^0$.
\end{proposition}

\begin{proof}
Suppose (a) holds, and write $f=\sum_{j=1}^n g_j h_j$, where each $g_j$ and $h_j$ is homogeneous of positive degree. Letting $\partial_i=\frac{\partial}{\partial x_i}$, we have
\begin{displaymath}
\partial_i(f) = \sum_{j=1}^n (\partial_i(g_j) h_j + g_j \partial_i(h_j)) \in (g_1, \ldots, g_n, h_1, \ldots, h_n).
\end{displaymath}
Thus (b) holds.

Now suppose (b) holds. Let $g_1, \ldots, g_n$ be positive degree homogeneous elements such that $\partial_i f \in (g_1, \ldots, g_n)$ for all $i$. Write $\partial_i f = \sum_{j=1}^n h_{i,j} g_j$. Then by Euler's identity, we have
\begin{displaymath}
f = \frac{1}{d} \sum_{i \ge 1} x_i \partial_i f = \sum_{j=1}^n g_j h_j,
\end{displaymath}
where $h_j = \tfrac{1}{d} \sum_{i \ge 1} x_i h_{i,j}$ and $d=\deg(f)$, and so $f$ has strength $\le n$.
\end{proof}

\begin{remark}
The proof given above for (a) $\implies$ (b) is valid in $\cR^{\flat}$. The proof of the converse is not valid in $\cR^{\flat}$, though: the problem is that there is no apparent reason for $\sum_{i \ge 1} x_i h_{i,j}$ to have bounded coefficients. However, the converse direction still holds in $\cR^{\flat}$, and can be proved using a variant of our next argument.
\end{remark}

\begin{proof}[Proof of Theorem~\ref{mainthm2}]
Let $f \in \cR^{\flat}$ be given such that $f$ has finite strength in $\cR$. We show that $f$ has finite strength in $\cR^{\flat}$. We may as well scale $f$ by a power of $t$ and assume that $f \in \cR^+$. We will in fact show that $f$ has finite strength in $\cR^+$.

Let $I \subset \cR^+$ be the ideal generated by the partial derivatives of $f$. Then $I\cR$ is the ideal of $\cR$ generated by the partial derivatives of $f$. Since $f$ has finite strength in $\cR$,  Proposition~\ref{prop:str} implies that $I\cR$ is contained in a finitely generated non-unital ideal. Since $\cR$ is a polynomial ring, it follows that $I\cR$ has finite height (Proposition~\ref{prop:finhtgen}). By Proposition~\ref{prop:htbd1}, we see that $I\cR^{\flat}$ has finite height. Since $\cR^{\flat}$ is a polynomial ring, it follows from Proposition~\ref{prop:finhtgen} that $I\cR^{\flat}$ is contained in a finitely generated non-unital ideal $J'=(g_1, \ldots, g_r)$. Since $I\cR^{\flat}$ is homogeneous, we can assume that the $g_i$'s are homogeneous of positive degree. Scale each $g_i$ by a power of $t$ if necessary so that $g_i \in \cR^+$.

Now, $I$ is contained in the contraction of $J'$ to $\cR^+$, which is exactly the saturation of the ideal $J$ of $\cR^+$ generated by $g_1, \ldots, g_r$. Thus $I \subset \Sat(J)$. If $f$ has degree $d$ then $I$ is generated by elements of degree $d-1$, and so we have $I \subset \Sat_{\le d-1}(J)$. This is finitely generated by Theorem~\ref{thm:sat}; note that, in the notation of \S \ref{s:sat}, if $S=\cR^0$ then $R \cong \cR^+$. Since $I$ has no non-zero degree~0 elements, the same is true for $\Sat_{\le d-1}(J)$. Thus $f$ has finite strength in $\cR^+$ by Proposition~\ref{prop:str}.
\end{proof}

\section{Heights in $\cR^+$} \label{s:Rplus}

We now prove some additional results about heights in the ring $\cR^+$. In what follows, we let $B_n=(\cR^+_{>n})_{(t)}$ be the localization of $\cR^+_{>n}$ at the prime ideal $(t)$. We note that these rings are all isomorphic to each other.

\begin{proposition} \label{prop:zerocontract2}
Let $I$ be an ideal of $\cR^+$ of finite height. Then $I \cap \cR^+_{>n} \subset t \cR^+_{>n}$ for $n \gg 0$.
\end{proposition}

\begin{proof}
Since $I$ is contained in a finite height prime, it suffices to treat the case where $I=\fp$ is itself prime. First suppose $t \in \fp$, and let $\ol{\fp}$ be the extension of $\fp$ to $\cR^0=\cR^+/t\cR^+$. Then $\ol{\fp}$ is prime and has finite height, as a chain of primes below $\ol{\fp}$ would give one below $\fp$, and thus has length bounded by the height of $\fp$. Thus by Proposition~\ref{prop:zerocontract}, we have $\ol{\fp} \cap \cR^0_{>n}=0$ for $n \gg 0$. This gives $\fp \cap \cR^+_{>n} \subset t \cR^+_{>n}$, as required.

Next suppose $t \not\in \fp$. Then $\fp$ is the contraction of a prime $\fq$ of $\cR^{\flat}=\cR^+[1/t]$, necessarily of finite height by Proposition~\ref{prop:localht}. Appealing to Proposition~\ref{prop:zerocontract} again, we have $\fq \cap \cR^{\flat}_{>n}=0$ for $n \gg 0$. This implies $\fp \cap \cR^+_{>n}=0$ for $n \gg 0$.
\end{proof}

\begin{corollary}
Let $\fp$ be a finite height prime of $\cR^+$. Then $\fp$ is the contraction of a prime of $B_n[x_1, \ldots, x_n]$ for all sufficiently large $n$.
\end{corollary}

\begin{proposition} \label{prop:B}
The ring $B_n$ is a DVR containing $\bC\lbb t \rbb$, and has $t$ for a uniformizer.
\end{proposition}

\begin{proof}
Let $f$ be a non-zero element of $B$. Then we can write $f=a/b$ where $a,b \in \cR^+_{>n}$ and $b \not\in t \cR^+$. Write $a=t^k a_0$ where $a_0 \not\in t\cR^+_{>n}$; note that $k$ is the minimal non-negative integer such that $t^k$ divides all coefficients of $a$. Then $f=t^k (a_0/b)$, and $a_0/b$ is a unit of $B_n$. Thus every non-zero element of $B_n$ has the form $ut^n$ for $u$ a unit, which proves the claim.
\end{proof}

\begin{proposition} \label{prop:cat}
Let $\fp \subset \fq$ be finite height primes of $\cR^+$. Then any two maximal chains of primes between $\fp$ and $\fq$ have the same length.
\end{proposition}

\begin{proof}
Let $\fp=\fa_0 \subset \cdots \subset \fa_r=\fq$ and $\fp=\fb_0 \subset \cdots \subset \fb_s=\fq$ be two maximal chains. Note that $r$ and $s$ are finite since they are bounded by $\hgt_{\cR^+}(\fq)$, which is finite. Let $n$ be sufficiently large so that the $\fa_i$ and $\fb_j$ are all contracted from $B_n[x_1, \ldots, x_n]$. The extensions of these two chains are both maximal chains between the extensions of $\fp$ and $\fq$ in $B_n[x_1, \ldots, x_n]$. They thus have the same length since $B_n[x_1, \ldots, x_n]$ is a catenary ring. (Any DVR is universally catenary, see \stacks{00NM}.)
\end{proof}

\begin{corollary} \label{cor:cat}
Let $\fp$ be a prime of $\cR^+$ of finite height. Then any maximal chain $\fp_0 \subset \cdots \subset \fp_r =\fp$ has length $r=\hgt_{\cR^+}(\fp)$.
\end{corollary}

\begin{proposition} \label{prop:htmodt}
Let $I$ be an ideal of $\cR^+$ of finite height that contains $t$, and let $\ol{I}=I \cR^0$. Then
\begin{displaymath}
\hgt_{\cR^+}(I)=\hgt_{\cR^0}(\ol{I})+1.
\end{displaymath}
\end{proposition}

\begin{proof}
First suppose that $I=\fp$ is prime. Let $c=\hgt_{\cR^0}(\ol{\fp})$, and let $0=\ol{\fp}_0 \subset \cdots \subset \ol{\fp}_c=\ol{\fp}$ be a maximal chain of primes. Let $\fp_{i+1}$ be the inverse image of $\fp_i$ in $\cR^+$; note that $\fp_1=(t)$. Put $\fp_0=0$. Then $\fp_0 \subset \cdots \subset \fp_{c+1} =\fp$ is a maximal chain of primes in $\cR^+$, and so $\hgt_{\cR^+}(\fp)=c+1$ by Corollary~\ref{cor:cat}.

Now let $I$ be an arbitrary ideal containing $t$ of height $c<\infty$, and let $d=\hgt_{\cR^0}(\ol{I})$. Let $\fp$ be a height $c$ prime of $\cR^+$ containing $I$. Then $\fp$ contains $t$, and so $\hgt_{\cR^0}(\ol{\fp})=c-1$. Since $\ol{I} \subset \ol{\fp}$, we find $d \le c-1$. Conversely, suppose that $\ol{\fp}$ is a height $d$ of $\cR^0$ containing $\ol{I}$. Then its inverse image $\fp$ has height $d+1$ and contains $I$, and so $c \le d+1$. This completes the proof.
\end{proof}

\begin{proposition}
The ring $\cR^+$ satisfies condition $(\ast)$ of Proposition~\ref{prop:supht}.
\end{proposition}

\begin{proof}
Let $I$ be an ideal of $\cR^+$ of height $c<\infty$. Let $S$ be the set of primes of $\cR^+$ of height $c$ that contain $I$. We must show that $S$ is finite. Let $S_1$ be the set of $\fp \in S$ such that $t \not\in \fp$, and let $S_2$ be the complement. We show that $S_1$ and $S_2$ are each finite.

Suppose $S_1$ is non-empty. We first claim that $\hgt_{\cR^{\flat}}(I\cR^{\flat})=c$. We have $c \le \hgt_{\cR^{\flat}}(I \cR^{\flat})$ by Corollary~\ref{cor:localht}. For $\fp \in S_1$ we have $I \cR^{\flat} \subset \fp \cR^{\flat}$ and $\fp \cR^{\flat}$ has height $c$ by Proposition~\ref{prop:localht}. Thus $\hgt_{\cR^{\flat}}(I \cR^{\flat}) \le c$, which proves the claim. The same reasoning shows that $S_1$ is in bijection with the set of height $c$ primes of $\cR^{\flat}$ containing $I \cR^{\flat}$. Since condition $(\ast)$ holds for $\cR^{\flat}$ (Proposition~\ref{prop:polystar}), it follows that $S_1$ is finite.

Suppose $S_2$ is non-empty. Let $J=I+(t)$. Then $I \subset J \subset \fp$ for any $\fp \in S_2$. Since $I$ and $\fp$ have height $c$, it follows that $J$ has height $c$. Let $\ol{J}=J \cR^0$, which has height $c-1$ by the Proposition~\ref{prop:htmodt}. By Proposition~\ref{prop:htmodt}, we see that $S_2$ is in bijection with the height $c-1$ primes of $\cR^0$ containing $\ol{J}$. Since $(\ast)$ holds for $\cR^0$ (Proposition~\ref{prop:polystar}), it follows that $S_2$ is finite.
\end{proof}

\begin{corollary} \label{cor:ht-lim-Rplus}
Let $I$ be an ideal of $\cR^+$ and let $I=\bigcup_{\alpha \in \cI} J_{\alpha}$ be a directed union. Then $\hgt_{\cR^+}(I)=\sup_{\alpha \in \cI} \hgt_{\cR^+}(J_{\alpha})$.
\end{corollary}

\begin{proof}
This follows from Proposition~\ref{prop:supht}.
\end{proof}

\begin{proposition}
Let $I$ be a finitely generated non-unital ideal of $\cR^+$. Then $I$ has finite height.
\end{proposition}

\begin{proof}
First suppose that $J=I+(t)$ is not the unit ideal. It is finitely generated, and so $\ol{J}=J \cR^0$ is finitely generated. By Proposition~\ref{prop:htmodt}, we have $\hgt_{\cR^+}(J)=\hgt_{\cR^0}(\ol{J})+1$. Since $\ol{J}$ is a finitely generated ideal in the polynomial ring $\cR^0$, it has finite height (Proposition~\ref{prop:finhtgen}). Thus $J$ has finite height, and so $I$ does as well.

Now suppose that $I+(t)$ is the unit ideal. Then $t$ is a unit in $\cR^+/I$, and so $I$ does not contain any power of $t$. It follows that $I\cR^{\flat}$ is a finitely generated non-unital ideal. It is thus finite height (Proposition~\ref{prop:htmodt}), and so is contained in a finite height prime $\fp$. Thus $I$ is contained in the contraction of $\fp$, which has finite height (Proposition~\ref{prop:localht}), and so $I$ has finite height.
\end{proof}

\begin{proposition}[Hauptidealsatz] \label{prop:hauptideal}
Let $I$ be an ideal of $\cR^+$ of finite height $c$ and let $f \in \cR^+$. Suppose that $I+(f)$ is not the unit ideal. Then $I+(f)$ has height $\le c+1$.
\end{proposition}

\begin{proof}
First suppose that $I$ is finitely generated. Then $J=I+(f)$ is also finitely generated, and thus has finite height by the previous proposition. Let $\fp$ be a height $c$ prime containing $I$ and let $\fq$ be a finite height prime containing $J$. Let $n$ be such that $\fp$ and $\fq$ are contracted from $B_n[x_1, \ldots, x_n]$. The extension $I'$ of $I$ to $B_n[x_1, \ldots, x_n]$ has height $c$: indeed, it has height at least $c$ (Corollary~\ref{cor:localht}) and is contained in the extension of $\fp$, which has height $c$ (Proposition~\ref{prop:localht}). Furthermore, $I'+(f)$ is not the unit ideal of $B_n[x_1, \ldots, x_n]$, as it is contained in the extension of the prime $\fq$. Thus, by the classical Hauptidealsatz, $I'+(f)$ has height $\le c+1$. Let $\fP$ be a prime of $B_n[x_1, \ldots, x_n]$ of height $\le c+1$ containing $I'+(f)$. Then the contraction of $\fP$ to $\cR^+$ has height $\le c+1$ (Proposition~\ref{prop:localht}) and contains $J=I+(f)$. Thus $\hgt_{\cR^+}(J) \le c+1$.

We now treat the general case. Write $I=\bigcup_{\alpha \in \cI} J_{\alpha}$ (directed union) with $J_{\alpha}$ a finitely generated ideal contained in $I$. By Corollary~\ref{cor:ht-lim-Rplus}, $\hgt_{\cR^+}(J_{\alpha})=c$ for $\alpha$ sufficiently large; we may as well assume it holds for all $\alpha$ by passing to a cofinal subset. We have $I+(f) = \bigcup_{\alpha \in \cI} (J_{\alpha}+(f))$. Thus, applying Corollary~\ref{cor:ht-lim-Rplus} again, we have $\hgt_{\cR^+}(I+(f)) = \sup_{\alpha \in \cI} \hgt_{\cR^+}(J_{\alpha}+(f))$. By the first paragraph, we have $\hgt_{\cR^+}(J_{\alpha}+(f)) \le c+1$ for all $\alpha$. It follows that $\hgt_{\cR^+}(I+(f)) \le c+1$, which completes the proof.
\end{proof}

\end{document}